\documentclass[leqno,12pt]{amsart} 
\usepackage{amssymb}

\usepackage{scalerel,stackengine}
\stackMath
\newcommand\reallywidehat[1]{%
\savestack{\tmpbox}{\stretchto{%
  \scaleto{%
    \scalerel*[\widthof{\ensuremath{#1}}]{\kern-.6pt\bigwedge\kern-.6pt}%
    {\rule[-\textheight/2]{1ex}{\textheight}}
  }{\textheight}%
}{0.5ex}}%
\stackon[1pt]{#1}{\tmpbox}%
}

\theoremstyle{plain} 
\newtheorem{theorem}{\indent\sc Theorem}[section]
\newtheorem{lemma}[theorem]{\indent\sc Lemma}
\newtheorem{corollary}[theorem]{\indent\sc Corollary}
\newtheorem{proposition}[theorem]{\indent\sc Proposition}

\theoremstyle{definition} 

\newtheorem*{theoreme*}{\indent\sc Theorem}
\usepackage{lmodern}
\usepackage[all]{xy}

\newcommand{\R}{\mathbb{R}}
\renewcommand{\H}{\mathbb{H}}
\newcommand{\C}{\mathbb{C}}

\newcommand{\sn}{\mathbb{S}}

\newcommand{\I}{\mathbb{J}_{\C\mathrm{P}^2}}

\newcommand{\hh}{\mathcal{H}}
\newcommand{\vv}{\mathcal{V}}

\newcommand{\id}{\mathrm{id}}
\newcommand{\scal}{\mathrm{scal}}
\newcommand{\Vect}{\mathrm{Vect}}
\newcommand{\Span}{\mathrm{span}}
\newcommand{\Hom}{\mathrm{Hom}}

\newcommand{\F}{\mathbb{F}_{1,2}}
\newcommand{\sq}{\mathbb{S}^{4}}
\newcommand{\CP}{\mathbb{C}\mathrm{P}}
\newcommand{\CH}{\mathbb{C}\mathrm{H}}
\newcommand{\It}{\mathbb J_2}
\newcommand{\Io}{\mathbb J_1}
\newcommand{\so}{\mathfrak{so}}
\newcommand{\SO}{{\rm SO}}
\newcommand{\U}{{\rm U}}

\begin{document}

\title{Hypersurfaces of the nearly K{\"a}hler twistor spaces $\CP^3$ and $\F$}
\author{Guillaume~Deschamps}
\author{Eric~Loubeau}


\address{University of Brest, CNRS UMR 6205, LMBA, F-29238 Brest, France}
\email{guillaume.deschamps@univ-brest.fr, eric.loubeau@univ-brest.fr}

\begin{abstract}
    In this article, we show that a hypersurface of the nearly K{\"a}hler $\CP^3$ or $\F$ cannot have its shape operator and induced almost contact structure commute together.
    This settles the question for six-dimensional homogeneous nearly K{\"a}hler manifolds, as the cases of $\sn^6$ and $\sn^3 \times \sn^3$ were previously solved, and provides a counterpart to the more classical question for the complex space forms $\CP^n$ and $\CH^n$.
    The proof relies heavily on the construction of $\CP^3$ and $\F$ as twistor spaces of $\sq$ and $\CP^2$.
    \end{abstract}

\maketitle

\section*{Introduction}

The classical study of real hypersurfaces in complex space forms has led to extensive lists by Takagi~\cite{Takagi1,Takagi2} (for $\CP^n$) and Montiel~\cite{Montiel} (for $\CH^n$).

Driven by the number of principal curvatures and the importance of Hopf hypersurfaces, i.e. when the ambient complex structure maps the normal vector field to a principal direction, hypersurfaces of $\CP^n$ or $\CH^n$ where the shape operator $A$ and the induced almost contact structure $\varphi$ commute constitute a remarkable class, amenable to classification.

Indeed, by \cite[Th. 6.19]{Cecil2015} their principal curvatures must be constant and in twos or threes. Moreover, they must belong to type A of the Takagi-Montiel lists (cf. \cite[Th. 8.37]{Cecil2015} as well)

An almost Hermitian manifold $(Z,\mathbb J,g)$ is called nearly K{\"a}hler \cite{Gray1970} if $\nabla \mathbb J$ is anti-symmetric. The best-known (non-K{\"a}hler) example is the round sphere $\sn^6$ with its canonical metric and the structure that comes from octonion multiplication.

In view of the classical theory for complex space forms, it is natural to ask which hypersurfaces of nearly K{\"a}hler manifolds satisfy $A\varphi =\varphi A$.

Nearly K{\"a}hler manifolds enjoy many topological and geometric properties akin to K{\"a}hler geometry (cf.~\cite{Gray1976}) and have known a recent revival of interest with the structure theorem of Nagy~\cite{Nagy2002} in 2002, which shows that six-dimensional nearly K{\"a}hler manifolds act as building blocks, and Butruille's 2005 classification of homogeneous nearly K{\"a}hler six-manifolds~\cite{Butruille2005}, namely $\sn^6$, $\sn^3 \times \sn^3$, $\CP^3$ and $\F$.

While the explicit construction of the nearly K{\"a}hler structure (and metric) on $\sn^3 \times \sn^3$ is rather involved and ad-hoc, the $\CP^3$ and $\F$ examples both have their origin in twistor theory, as twistor spaces of $\sn^4$ and $\CP^2$.

In the case of the four-dimensional sphere, as its unitary frame bundle is $\SO(5)$, its twistor space is the associated bundle
$$
\SO(5) \times_{\SO(4)} \SO(4)/\U(2) \simeq \SO(5)/\U(2)
$$
which is $\CP^3$ and the twistor projection $\CP^3 \to \sn^4$ is $\Span_{\C} v \mapsto \Span_{\H} v$, where $\sn^4 \simeq \H\mathrm{P}^1$ by the Hopf map. When the spaces are equipped with their canonical metrics, this projection is a Riemannian submersion.

For the two-dimensional complex projective space, one considers $y\in \CP^2$ and $(x,y,z)$ mutually orthogonal complex lines in $\C^3$. Identifying a complex structure in $T_{y} \CP^2$ with a choice of holomorphic and anti-holomorphic bundles, one shows that
\begin{equation}\label{hom1}
    T^{1,0}_{y} \CP^2 = \Hom(y,x) \oplus \Hom(y,z)
\end{equation}
and 
\begin{equation}\label{hom2}
    T^{0,1}_{y} \CP^2 = \Hom(x,y) \oplus \Hom(z,y) .
\end{equation}
 There is a one-one correspondence between triples $(x,y,z)$ and couples $(l,p)$, where $l$ is a complex line and $p$ a complex plane in $\C^3$ with $l \subset p$, i.e. the flag manifold $\F$.
 
 Since $\CP^2$ is self-dual \cite{Atiyah}, the integrable almost Hermitian structure is defined by taking the standard Hermitian structure on $T^{1,0}_{y} \CP^2$ and its opposite on $T^{0,1}_{y} \CP^2 $. Because of this orientation reversal, this identifies $\F$ with $Z(\overline{\CP^2})$, and the twistor projection $(l\subset p) \in \F \mapsto l^\perp \cap p \in \CP^2$ is also a Riemannian submersion.

 There is a general procedure, due to \cite{Eells-Salamon} and \cite{Nagy2002}, to produce nearly K{\"a}hler manifolds:
 If $(Z,\Io , g_1)$ is a K{\"a}hler manifold with a Riemannian foliation $\mathcal{F}$, which induces an ($\Io$-invariant) integrable distribution $\vv$ and its orthogonal complement $\hh$, then the Riemannian metric
 $$
 g_2 (X,Y) = \frac12 g_1 (X,Y) \quad \forall X,Y \in \vv
 $$
 and
  $$
 g_2 (X,Y) =  g_1 (X,Y) \quad \forall X,Y \in \hh
 $$
 together with the almost complex structure
 $$
 \It X = - \Io X \quad \forall X \in \vv \qquad \text{and} \qquad 
 \It X =  \Io X \quad \forall X\in \hh
 $$
 make $(Z,\It ,g_2)$ into a nearly K{\"a}hler manifold.

 According to Hitchin~\cite{Hitchin}, $\CP^3$ and $\F$ are the only compact twistor spaces $(Z,\Io , g_1)$ to be K{\"a}hler. Therefore, if $(Z,\It , g_2)$ is nearly K{\"a}hler then $Z$ is $\CP^3$ or $\F$.

Let $(Z,\It,g)$ be a nearly K{\"a}hler manifold and $H \hookrightarrow Z$ a hypersurface. Call $N$ the unit normal to $H$ and then define an almost contact (metric) structure $\varphi$ on $H$ by:
$$\varphi X = \It X - g(\It X,N)N, \quad \forall X\in TH.$$
One easily verifies that 
$$ g(\varphi X, \varphi Y) = g(X,Y)\quad \forall X,Y\in TH\cap (\It N)^\perp,$$
or more generally
$$ g(\varphi X, Y) = g(X, \varphi Y)\quad \forall X,Y\in TH,$$
as well as
$$ \varphi (\It N)=0.$$
The other fundamental tensor is the shape operator $A$ of $H$:
$$ AX = - \nabla_{X}^{Z} N ,$$
so that
$$ \nabla_{X}^{Z} Y = \nabla_{X}^{H} Y + g(AX,Y)N .$$

An immediate remark on hypersurfaces which satisfy $A\varphi =\varphi A$ is that the Hopf vector field $\It N$ has to be an eigenvector of $A$ (of eigenvalue $\mu$), so that $H$ is a Hopf hypersurface, and the eigenspaces of $A|_{(\It N)^{\perp}}$ must be $\It$-stable.

In dimension $6$, since we have a full classification of homogeneous nearly K{\"a}hler~\cite{Butruille2005}, the first two cases of the list, $\sn^6$ and $\sn^3 \times \sn^3$, have already been investigated.

For $\sn^6$, we combine \cite[Proposition 1]{Berndt1995}, which shows that the Hopf vector field must have geodesic integral curves, arguments from Lemma~\ref{lemma2} to deduce that it is in fact a Killing vector field and conclude with \cite[Theorem 2]{Martins2001} that the only hypersurfaces of $\sn^6$ with $A\varphi =\varphi A$ are (open parts of) geodesic spheres.

For the nearly K{\"a}hler $\sn^3 \times \sn^3$, that is equipped with the right metric and almost complex structure, its hypersurfaces with $A\varphi =\varphi A$ must be locally given by the canonical immersion of $\sn^2 \times \sn^3$ in $\sn^3 \times \sn^3$ (\cite[Theorem 1.3]{Hu2017}). Note that this  classification contains three immersions but, by \cite[p. 155]{Moruz2018}, they turn out to be all isometric one to the other.

There exists an almost contact counterpart to the nearly K{\"a}hler condition, coined nearly cosymplectic (and defined by $\nabla\varphi$ being antisymmetric). By \cite{Blair1971}, they must satisfy $A\varphi =\varphi A$ but, while $\sn^5 \hookrightarrow \sn^6$ is well-known to be nearly cosymplectic, the hypersurface $\sn^2 \times \sn^3 \hookrightarrow\sn^3 \times \sn^3$ is not, as a quick inspection of the eigenvalues of its shape operator reveals.

The objective of this article is to extend these results to the remaining two homogeneous nearly K{\"a}hler six-manifolds and prove the following theorem.
\begin{theoreme*}
Let $Z(M)$ be the nearly K{\"a}hler manifold $\CP^3$ or $\F$. Then there exists no hypersurface $H \hookrightarrow  Z(M)$ such that its shape operator $A$ and the induced almost contact structure $\varphi$ commute: 
$$A\varphi =\varphi A .$$
\end{theoreme*}

A direct consequence is that this construction produces only one example of nearly cosymplectic almost contact hypersurface.

\begin{corollary}
The only nearly cosymplectic hypersurface of a homogeneous $6$-dimensional nearly K{\"a}hler manifold is $\sn^5 \hookrightarrow \sn^6$.
\end{corollary}

As a byproduct of the Theorem, we obtain information on the eigenvalues of the shape operator $A$.
\begin{corollary}
There is no totally geodesic or totally umbilical hypersurface of the nearly K{\"a}hler manifolds $\CP^3$ or $\F$.
\end{corollary}

The authors thank the anonymous referees for their careful reading of the manuscript and valuable suggestions.

\section{Curvature properties of nearly K{\"a}hler $\CP^3$ and $\F$}

Throughout the rest of this article, we specialize to the cases $M=\sq$ and $M=\overline{\CP^2}$. Let $Z(M)$ be the twistor space of $M$, equiped with the Riemannian metric \cite{Apostolov98} (see also \cite{Friedrich} for an earlier reference)
$$g_t = \pi^* g_M + t g_{\CP^1}, \quad (t>0) .$$
Two almost complex structures can be defined on $Z(M)$: First the Atiyah-Hitchin-Singer structure $\Io$ on $T_{(x_0,I)} Z(M)$, with $x_0 \in M$ and $I$ a complex structure on 
$T_{x_0} M$, defined by
$$
\left\{\begin{array}{ccc}
\Io X &=& IX \quad \text{if } X\in \hh\\
\Io Y&=&  \mathbb{J}_{\C\mathrm{P}^1} Y,\quad \text{if } Y\in \vv ,
\end{array}\right.
$$
where we identify vectors tangent to $M$ with their horizontal lifts in $\hh \subset T_{(x_0,I)} Z(M)$;
Second the Eells-Salamon structure~\cite{Eells-Salamon}:
$$
\left\{\begin{array}{ccc}
\It &=& \Io \quad \text{on } \hh\\
\It &=& - \Io \quad \text{on } \vv.
\end{array}\right.
$$

Then, as the cases we consider are anti-self dual, \cite{Friedrich} shows that $(Z(M), g_t, \Io)$ is a K{\"a}hler manifold for $t= \frac{12}{s}$, ($s = \scal_{(M,g_M)}$), while~\cite{Mush} prove that $(Z(M),g_t, \It)$ is  nearly K{\"a}hler for $t= \frac{6}{s}$.

The next proposition relates the curvature tensors of the twistor space and the base manifold, in terms of the nearly K{\"a}hler structure. This will lead to crucial curvature properties in Lemma~\ref{lemma2}.
\begin{proposition}\label{prop1}
Let $Z(M)$ be the twistor space of $\sq$ or $\CP^2$. Write $g= g_{\tfrac{6}{s}}$ so that $(Z(M),\It, g)$ is nearly K{\"a}hler and denote by $R$ and $R^M$ the respective curvature tensors of $(Z(M),g)$ and $(M,g_M)$.

Let $X,Y,Z,T \in T_p Z(M)$ ($p\in Z(M)$) then 

\begin{align*}
&R(X,Y,Z,T)=R^M\Big(d\pi(X),d\pi(Y),d\pi(Z),d\pi(T)\Big)\\
&+2(b+2a)g^h(\mathbb J_2X,Y)g^h(\mathbb J_2Z,T)
+(b+2a)g^h(\mathbb J_2X,Z)g^h(\mathbb J_2Y,T)\\
&-(b+2a)g^h(\mathbb J_2X,T)g^h(\mathbb J_2Y,Z)
+(c-5b)\Big(g^h(X,Z)g^h(Y,T)-g^h(X,T)g^h(Y,Z)\Big)\\
&-2ag^h(\mathbb J_2X,Y)g(\mathbb J_2 Z,T)
-2a g(\mathbb J_2X,Y)g^h(\mathbb J_2Z,T)
-ag^h(\mathbb J_2X,Z)g(\mathbb J_2Y,T)\\
&-ag(\mathbb J_2X,Z)g^h(\mathbb J_2Y,T)
+ag^h(\mathbb J_2X,T)g(\mathbb J_2Y,Z)+ag(\mathbb J_2X,T)g^h(\mathbb J_2Y,Z)\\
&+(b-c)\Big(g^h(X,Z)g(Y,T)+g(X,Z)g^h(Y,T) -g^h(X,T)g(Y,Z)-g(X,T)g^h(Y,Z)\Big)\\
&+c\Big(g(X,Z)g(Y,T)-g(X,T)g(Y,Z)\Big),
\end{align*}
where 
\begin{align*}
a&= \frac{s}{24}-t\Big(\frac{s}{24}\Big)^2 ; b= t\Big(\frac{s}{24}\Big)^2 ;c= \frac1t ,
\end{align*}
with $t = \frac{6}{s}$
\end{proposition}
\begin{proof}
We rely on the formula of \cite{Apostolov98}:
\begin{align*}
&R(X,Y,Z,T)=R^M\Big(d\pi(X),d\pi(Y),d\pi(Z),d\pi(T)\Big)\\
&+2ag^h(\mathbb J_1X,Y)g^v(\mathbb J_1Z,T)+2ag^v(\mathbb J_1X,Y)g^h(\mathbb J_1Z,T)
+ag^h(\mathbb J_1X,Z)g^v(\mathbb J_1Y,T)\\
&+ag^v(\mathbb J_1X,Z)g^h(\mathbb J_1Y,T)
-ag^h(\mathbb J_1X,T)g^v(\mathbb J_1Y,Z)-ag^v(\mathbb J_1X,T)g^h(\mathbb J_1Y,Z)\\
&+2bg^h(\mathbb J_1X,Y)g^h(\mathbb J_1Z,T)+bg^h(\mathbb J_1X,Z)g^h(\mathbb J_1Y,T)
-bg^h(\mathbb J_1X,T)g^h(\mathbb J_1Y,Z)\\
&+bg^h(X,Z)g^v(Y,T)+bg^v(X,Z)g^h(Y,T)
-bg^h(X,T)g^v(Y,Z)-bg^v(X,T)g^h(Y,Z)\\
&-3bg^h(X,Z)g^h(Y,T)+cg^v(X,Z)g^v(Y,T)
+3bg^h(X,T)g^h(Y,Z)-cg^v(X,T)g^v(Y,Z),
\end{align*}
where $\mathbb J_1$ is the K{\"a}hler structure on $Z(M)$. 

Since $\Io$ and $\It$ agree on the horizontal distribution and are opposite on $\vv$, we have

\begin{align*}
&R(X,Y,Z,T)=R^M\Big(d\pi(X),d\pi(Y),d\pi(Z),d\pi(T)\Big)\\
&-2ag^h(\mathbb J_2X,Y)g^v(\mathbb J_2Z,T)-2ag^v(\mathbb J_2X,Y)g^h(\mathbb J_2Z,T)
-ag^h(\mathbb J_2X,Z)g^v(\mathbb J_2Y,T)\\
&-ag^v(\mathbb J_2X,Z)g^h(\mathbb J_2Y,T)
+ag^h(\mathbb J_2X,T)g^v(\mathbb J_2Y,Z)+ag^v(\mathbb J_2X,T)g^h(\mathbb J_2Y,Z)\\
&+2bg^h(\mathbb J_2X,Y)g^h(\mathbb J_2Z,T)
+bg^h(\mathbb J_2X,Z)g^h(\mathbb J_2Y,T)
-bg^h(\mathbb J_2X,T)g^h(\mathbb J_2Y,Z)\\
&+bg^h(X,Z)g^v(Y,T)+bg^v(X,Z)g^h(Y,T)
-bg^h(X,T)g^v(Y,Z)-bg^v(X,T)g^h(Y,Z)\\
&-3bg^h(X,Z)g^h(Y,T)+cg^v(X,Z)g^v(Y,T)
+3bg^h(X,T)g^h(Y,Z)-cg^v(X,T)g^v(Y,Z).
\end{align*}

We use the shorthand $g^h(X,Y) = g(X^h , Y^h)$ and $g^v(X,Y) = g(X^v , Y^v)$ for $X= X^h + X^v$ its decomposition in the horizontal and vertical distributions.
Since $g(X,Y) = g^h(X,Y) + g^v(X,Y)$, we have 
\begin{align*}
&R(X,Y,Z,T)=R^M\Big(d\pi(X),d\pi(Y),d\pi(Z),d\pi(T)\Big)\\
&+4ag^h(\mathbb J_2X,Y)g^h(\mathbb J_2Z,T)+2ag^h(\mathbb J_2X,Z)g^h(\mathbb J_2Y,T)
-2ag^h(\mathbb J_2X,T)g^h(\mathbb J_2Y,Z)\\
&-2ag^h(\mathbb J_2X,Y)g(\mathbb J_2 Z,T)-2a g(\mathbb J_2X,Y)g^h(\mathbb J_2Z,T)
-ag^h(\mathbb J_2X,Z)g(\mathbb J_2Y,T)\\
&-ag(\mathbb J_2X,Z)g^h(\mathbb J_2Y,T)
+ag^h(\mathbb J_2X,T)g(\mathbb J_2Y,Z)+ag(\mathbb J_2X,T)g^h(\mathbb J_2Y,Z)\\
&+2bg^h(\mathbb J_2X,Y)g^h(\mathbb J_2Z,T)
+bg^h(\mathbb J_2X,Z)g^h(\mathbb J_2Y,T)
-bg^h(\mathbb J_2X,T)g^h(\mathbb J_2Y,Z)\\
&-2bg^h(X,Z)g^h(Y,T)+2bg^h(X,T)g^h(Y,Z)
+bg^h(X,Z)g(Y,T)+bg(X,Z)g^h(Y,T)\\
&-bg^h(X,T)g(Y,Z)-bg(X,T)g^h(Y,Z)+(c-3b)g^h(X,Z)g^h(Y,T)\\
&-(c-3b)g^h(X,T)g^h(Y,Z)\\
&+c(g(X,Z)g(Y,T)-g(X,Z)g^h(Y,T)-g^h(X,Z)g(Y,T))\\
&-c(g(X,T)g(Y,Z)-g(X,T)g^h(Y,Z)-g^h(X,T)g(Y,Z)), 
\end{align*}
and reorganising terms yields the proposition.
\end{proof}

Let $H \hookrightarrow Z(M)$ be a hypersurface of $(Z(M),\It,g)$ satisfying
\begin{equation}\label{eq-comm}
    A\varphi = \varphi A.
\end{equation}
We call $N$ the normal to $H$ and Equation~\eqref{eq-comm} implies that $\It N$ is an eigenvector of $A$ (of eigenvalue $\mu$). We denote by $\lambda$ an eigenvalue of $A$ and observe that the eigenspace $E_\lambda\cap (\mathbb J_2 N)^\perp$ is  $\It$-invariant.

\begin{lemma}\label{lemma2}
Let $X\in E_\lambda \cap (\mathbb J_2 N)^\perp$ then 
$$R(\It N,X,\It X,N) = - R(\It N,\It X,X,N),$$
and $R(\It N,X,\It X,N) = -\lambda (\lambda - \mu) \Vert X \Vert^2$.
\end{lemma}
\begin{proof}
Since both $X$ and $\It X$ belong to $E_\lambda$, the Codazzi Equation gives
\begin{align*}
    R(\It N,X,\It X,N) &= g((\nabla_{\It N}A)(X)-(\nabla_{X} A)(\It N),\It X)\\
    &= -g((\nabla_{X} A)(\It N),\It X)
\end{align*}
since $A|_{E_\lambda} = \lambda \id_{E_\lambda}$.

Therefore
$$ R(\It N,X,\It X,N) = (\lambda - \mu) g(\nabla_X \It N, \It X),$$
which is $\It $-invariant since
\begin{align*}
    \lambda \Vert X\Vert^2 = g(AX,X) &= - g( \nabla_X N , X) = g ( (\nabla_{X} \It )(\It N) + \It  \nabla_{X} \It N , X ) \\
    &= - g ( (\nabla_{\It N} \It )(X), X) - g(\nabla_X \It N , \It X) \\
    &= - g(\nabla_X \It N , \It X).
\end{align*}
\end{proof}

Motivated by the results of Lemma~\ref{lemma2}, we use Proposition~\ref{prop1} to obtain the following curvature expression.
\begin{corollary}\label{cor1}
If $X$ is a vector field in $(N,\mathbb J_2 N)^\perp$, we have
\begin{align*}
&R(\mathbb J_2 N,\mathbb J_2X,X,N)=R^M\Big(d\pi(\mathbb J_2 N),d\pi(\mathbb J_2 X),d\pi(X),d\pi(N)\Big)\\
&+(b+2a)g^h(N,X)^2+(-4a+3b-c)g^h(N,\mathbb J_2 X)^2
+a(\Vert N^h \Vert^2 \Vert X\Vert^2+\Vert X^h\Vert^2)\\
&-(2a+b)\Vert N^h \Vert^2\Vert X^h\Vert^2.
\end{align*}
\end{corollary}

From this symmetry of the curvature tensor, we can eliminate vertical normal vector fields.
\begin{proposition}\label{proppage6}
Let $H$ be a hypersurface of $Z(M)$ such that $A\varphi = \varphi A$. Then the normal vector $N$ cannot be vertical.
\end{proposition}
\begin{proof}
If $N$ is vertical then so is $\It N$ and all eigenvectors orthogonal to it must be horizontal.
But, for such an eigenvector $X$ associated to the eigenvalue $\lambda$ and orthogonal to $\It N$,
\begin{align*}
\lambda \Vert X\Vert^2& =g(AX,X)=g(-\nabla_{X}N,X)=g(N,\nabla_{X}X)\\
&=g(N,\frac12 [X,X])=0,
\end{align*}
by O'Neill~\cite{Oneill}, since $X$ is horizontal.
However, this implies, by Lemma~\ref{lemma2}, that  $R(\mathbb J_2N,\mathbb J_2X,X,N)$ is zero, that is, by $N^h=0$ and  Corollary~\ref{cor1},
$$a \Vert X^h\Vert^2 =0 ,$$
which contradicts the fact that $X$ is horizontal, since $a\neq 0$.
\end{proof}

\section{The case $(M,Z(M))= (\sq , \CP^3)$}

As $\scal_{\sq} = 12$, the constants in Proposition~\ref{prop1} take on the values $ a= \frac{3}{8}, b=\frac{1}{8}$ and $c= 2$. Since
$$
R^{\sq}(U,V,W,T)=g(V,W)g(U,T)-g(U,W)g(V,T) \quad \forall U,V,W,T \in T_p \sq,
$$
we have
\begin{align*}
\pi^* R^{\sq}(\mathbb J_2 N,\mathbb J_2 X, X,N)&=g^h(\mathbb J_2 X,N)^2 .
\end{align*}

From Corollary~\ref{cor1}, we obtain that for any $X\in (N,\mathbb J_2 N)^\perp$
\begin{align*}
R(\mathbb J_2 N,\mathbb J_2 X, X,N) &=
\frac{7}8g^h(X,N)^2-\frac{17}8 g^h(\mathbb J_2 X,N)^2
+\frac38\Big(\Vert X^h\Vert^2+\Vert N^h\Vert^2\Vert X\Vert^2\Big)\\
&-\frac78\Vert X^h\Vert^2\Vert N^h\Vert^2
\end{align*}
and
\begin{align*}
- R(\mathbb J_2 N,X,\mathbb J_2 X,N) &=
-\frac{17}8g^h(X,N)^2+\frac{7}8 g^h(\mathbb J_2 X,N)^2
+\frac38\Big(\Vert X^h\Vert^2+\Vert N^h\Vert^2\Vert X\Vert^2\Big)\\
&-\frac78\Vert X^h\Vert^2\Vert N^h\Vert^2 .
\end{align*}
By Lemma~\ref{lemma2}, when $X\in E_\lambda \cap (\mathbb J_2 N)^\perp$
$$
g^h(X,N)^2=g^h(\mathbb J_2X, N)^2.
$$
As this must remain true for the eigenvector $X+\mathbb J_2 X$, we infer that
\begin{equation}\label{Eq}
    g^h(X,N)=g^h(\mathbb J_2X, N)=0.
\end{equation}

We then easily prove that the vertical component of the normal vector field must be zero.

\begin{proposition}\label{lemmaA}
Let $H$ be a hypersurface of $\CP^3$ such that $A\varphi = \varphi A$. Then the normal vector field must be horizontal.
\end{proposition}

\begin{proof}
If $N^v\neq 0$, then $(N^v ,\It N^v)$ is a basis of the vertical distribution. But Equation~\eqref{Eq} forces
$$g^v(X,N)=g^v(\mathbb J_2X, N)=0$$ 
as $X$ is an eigenvector orthogonal to $N$ and $\It N$, so $X$ must be horizontal. Since this applies to all eigenvectors of $A$ in $(\mathbb J_2 N)^\perp$, they must be horizontal and orthogonal to $N$, hence $N^h$ must vanish, and we conclude with Proposition~\ref{proppage6}.
\end{proof}

The complementary contingency is resolved using tools from twistor theory.
\begin{proposition}\label{lemmaB}
Let $H$ be a hypersurface of $\CP^3$ such that $A\varphi = \varphi A$. Then $N$ cannot be horizontal.
\end{proposition}
\begin{proof}
If $N^v =0$, then for any horizontal $X$, we have, by O'Neill,
$$ (AX)^v = (-\nabla_X N )^v = -\frac12 ([X,N])^v.$$
Let $p= (x_0 ,I) \in H \subset \CP^3 = Z(\sq)$, $x_0\in \sq$ and $I$ a complex structure on $T_{x_0}\sq$. Take a positive orthonormal frame $(e_1,e_2,e_3,e_4)$ of $T_{x_0}\sq$ such that, at $p$:
$$
  e_1= d\pi(N), e_2 = I e_1, e_3\in (e_1,e_2)^\perp , e_4=I e_3.  
$$
Let $\vv_p$ be the vertical space at $p\in \CP^3$, i.e. the tangent space to the fibre.
We identify $\bigwedge^2 T_{x_0}\sq$ with $\so (T_{x_0}\sq)$, then there exists a surjection \cite{deBartolomeis}
\begin{align*}
 \so (T_{x_0}\sq) &\to \vv_{(x_0,I)}\\
P &\mapsto \widehat P:=[I,P]=IP-PI.
\end{align*}
Denote by $(I^+,J^+,K^+,I^-,J^-,K^-)$ the basis of $\bigwedge^2 T_{x_0}\sq$ with
\begin{center}\label{system}
    \(
        \left\{
\begin{array}{ccc}
I^+&=&e_1\wedge e_2+e_3\wedge e_4\\
J^+&=&e_1\wedge e_3-e_2\wedge e_4\\
K^+&=&e_1\wedge e_4+e_2\wedge e_3
\end{array}
\right.
    \)
  \hspace{.2in}  and \hspace{.2in}
    \(
        \left\{
\begin{array}{ccc}
I^-&=&e_1\wedge e_2-e_3\wedge e_4\\
J^-&=&e_1\wedge e_3+e_2\wedge e_4\\
K^-&=&-e_1\wedge e_4+e_2\wedge e_3
\end{array}
\right.
    \)
\end{center}
so that $I^+ = I$. From \cite{Davidov1991,deBartolomeis}, we know that 
$$ (AX)^v =-\frac12 [X,N]^v = -\frac12 \reallywidehat{R^{\sq}(X\wedge N)}.$$
One can easily check that
$$ R^{\sq}(e_3 \wedge N) = R^{\sq}(e_3 \wedge e_1) = e_1 \wedge e_3 = \frac12 (J^+ + J^-),$$
so
$\reallywidehat{R^{\sq}(e_3 \wedge N)} = K^+$.

Identifying $e_2$, $e_3$ and $e_4$ with their horizontal lifts, we have
$(Ae_3)^v = - \frac12 K^+$. Similarly $(Ae_4)^v = \frac12 J^+ .$

Therefore the block matrix of $A$ in the basis $\Big(\{e_2\},\{ e_3,e_4\},\{J^+,K^+\}\Big)$  is
$$
A=\left(\begin{array}{ccc}\mu &0&0\\0&E&F\\0&F&G\end{array}\right) 
\textrm{ with } F=-\frac12 \left(\begin{array}{ccc}0&-1\\1&0\end{array}\right),
$$
while $\varphi$ is
$$
\left(\begin{array}{ccc}0&0&0\\0&I&0\\0&0&-I\end{array}\right) \textrm{ with } I=\left(\begin{array}{ccc}0&-1\\1&0\end{array}\right) .
$$
As, by hypothesis $A\varphi = \varphi A$, a straightforward computation shows this to be impossible.
\end{proof}
Combining Propositions~\ref{lemmaA} and \ref{lemmaB} shows the $\CP^3$ case of the Theorem.

\section{The case $(M,Z(M))= (\overline{\CP^2} , \F)$}

The curvature tensor of $(\CP^2, g_{\CP^2}, \I)$ is 
\begin{align*}
R^{\CP^2}(U,V,W,S)&=g(U,S)g(V,W)-g(U,W)g(V,S)-g(U,\I W)g(V,\I S)\\
&+g(U,\I S)g(V,\I W)-2g(U,\I V)g(W,\I S)
\end{align*}
for $U,V,W$ and $S$ in $T_{x_0}\CP^2$ ($x_0 \in \CP^2$).
We still denote by $\mathbb J_{\CP^2}$ the almost complex structure induced on the horizontal distribution $\hh$, hence,
\begin{align*}
\pi^*R^{\CP^2}(\mathbb J_2N,\mathbb J_2X,X,N)&= g^h(N,\mathbb J_2X)^2 +
2g^h(\mathbb J_{\CP^2}N,X)^2-g^h(\mathbb J_{\CP^2}N,\mathbb J_2X)^2\\
&+g^h(\mathbb J_2N,\mathbb J_{\CP^2}N)g^h(\mathbb J_2X,\mathbb J_{\CP^2}X),
\end{align*}
and, as $\scal_{\CP^2} = 24$ and $t=\frac14$, $a=\tfrac34$, $b=\tfrac14$ and $c=4$.
From Corollary~\ref{cor1}, we obtain\label{page8}
\begin{align*}
&R(\mathbb J_2 N,\mathbb J_2X,X,N)=\tfrac74 g^h(N,X)^2 - \tfrac{21}{4} g^h(N,\mathbb J_2 X)^2 
+2g^h(\mathbb J_{\CP^2} N,X)^2 
-g^h(\mathbb J_{\CP^2}N,\mathbb J_2X)^2 \\
&+ \tfrac34 (\Vert N^h \Vert^2 \Vert X\Vert^2 +\Vert X^h\Vert^2)-\tfrac74 \Vert N^h \Vert^2\Vert X^h\Vert^2
+g^h(\mathbb J_2N,\mathbb J_{\CP^2}N)g^h(\mathbb J_2X,\mathbb J_{\CP^2}X).
\end{align*}
 We deduce, by Lemma~\ref{lemma2}:
\begin{lemma}\label{lemma*}
If $X\in E_{\lambda}$ then 
\begin{equation}\label{equation*}
7\Big(g^h(N,X)^2-g^h(N,\mathbb J_2 X)^2\Big)+3\Big(g^h(\mathbb J_{\CP^2} N,X)^2-g^h(\mathbb J_{\CP^2} N,\mathbb J_2 X)^2\Big)=0.
\end{equation}
\end{lemma}

The next result  is key to our argument since it reduces the type of the vector field normal to $H$ to just two possibilities.

\begin{proposition}\label{proposition2}
Let $H$ be a hypersurface of $\F$ such that $A\varphi = \varphi A$. Then the normal vector $N$ must be either vertical or horizontal.
\end{proposition}
\begin{proof}
The proof of Proposition~\ref{proposition2} consists of a series of lemmas.

Assume that $N$ is neither vertical nor horizontal. We consider a basis of the $T_p \CP^2$ given by
\begin{align*}
 e_1&= \frac{d\pi(N^h)}{\Vert N^h\Vert} ; e_2= I e_1 ;\\
 e_3 &= \begin{cases}
 \textrm{unitary part of $\I e_1$ that is normal to $(e_1,e_2)$, if non-zero,}\\
 \textrm{any unit vector in $(e_1,e_2)^\perp$, otherwise}
 \end{cases}.\\
 e_4&= I e_3 .
 \end{align*}
Recall that since $\CP^2$ is self-dual, $\F= Z(\overline{\CP^2})$ and $\I \in \bigwedge^{2}_{-}(\overline{\CP^2})$, so (using the same notation as on page~\pageref{system}) we can consider $\tilde{c},\tilde{s} \in \R$, with $\tilde{c}^2 + \tilde{s}^2 = 1$, such that $\I = \tilde{c} I^- + \tilde{s} J^-$, which, in the basis 
$(e_1,e_2,e_3,e_4)$, translates as
$$
I =I^+=\left(\begin{array}{rrrr}
0&-1&0&0\\
1&0&0&0\\
0&0&0&-1\\
0&0&1&0
\end{array}
\right),
\I=\left(\begin{array}{rrrr}
0&-\tilde{c}&-\tilde{s}&0\\
\tilde{c}&0&0&-\tilde{s}\\
\tilde{s}&0&0&\tilde{c}\\
0&\tilde{s}&-\tilde{c}&0
\end{array}
\right).
$$

We first describe the solutions to Equation~\eqref{equation*} in Lemma~\ref{lemma*} in the basis we just constructed.
\begin{lemma}\label{lemma13}
If $\tilde{s}\neq 0$, $d\pi(E_\lambda\cap (\mathbb J_2 N)^\perp)$ is included in
$$
\Vect\left(\left(\begin{array}{c}1\\0\\0 \\\delta_-\end{array}\right),\left(\begin{array}{c}0\\1\\-\delta_-\\0\end{array}\right)\right)
\bigcup
\Vect\left(\left(\begin{array}{c}1\\0\\0 \\\delta_+\end{array}\right),\left(\begin{array}{c}0\\1\\-\delta_+\\0\end{array}\right)\right),
 $$
with $\delta_\pm=\frac{6\tilde{c}\tilde{s}\pm\sqrt{84}\tilde{s}}{6\tilde{s}^2}$.\\
For $\tilde{s}=0$, $d\pi(E_\lambda\cap (\mathbb J_2 N)^\perp)$ is included in $\Vect(e_3, e_4)$.
\end{lemma}
\begin{proof}
Assume $X\in E_{\lambda}\cap (\mathbb J_2 N)^\perp$, with $X^h = (x,y,z,t)$ its coordinates in the basis $(e_1,e_2,e_3,e_4)$ (identifying vectors tangent to the base manifold with their horizontal lifts). Then, by Lemma~\ref{lemma*}, we have
\begin{align*}
0&=7\Big(g^h(X,N)^2-g^h(X,\mathbb J_2N)^2\Big)+3\Big(g^h(X,\I N)^2-g^h(X,\I \mathbb J_2 N)^2\Big)\\
&=7(x^2-y^2)+3\Big((\tilde{c}y+\tilde{s}z)^2-(-\tilde{c}x+\tilde{s}t)^2\Big)\\
&=(7-3\tilde{c}^2)(x^2-y^2)+3\tilde{s}^2(z^2-t^2)+6\tilde{c}\tilde{s}(xt+yz).
\end{align*}
Re-writing this system with the eigenvector $X+\mathbb J_2 X$, yields
\begin{align*}
0&=7g^h(X,N)g^h(X,\mathbb J_2N)+3g^h(X,\I N)g^h(X,\I \mathbb J_2N)\\
&=7xy+3(\tilde{c}y+\tilde{s}z)(-\tilde{c}x+\tilde{s}t)\\
&=(7-3\tilde{c}^2)xy+3\tilde{s}^2zt-3\tilde{c}\tilde{s}(xz-yt),
\end{align*}
so $X^h = (x,y,z,t)$ must satisfy the system
\begin{equation}\label{dagger}
\begin{cases}
(7-3\tilde{c}^2)(x^2-y^2)+3\tilde{s}^2(z^2-t^2)+6\tilde{c}\tilde{s}(xt+yz)=0 ,\\
  (7-3\tilde{c}^2)xy+3\tilde{s}^2zt-3\tilde{c}\tilde{s}(xz-yt)=0 .
  \end{cases}
  \end{equation}
We work with complex numbers $z_1=x+iy$ and $z_2=z+it$ to re-write \eqref{dagger} as a polynomial in $z_2$:
 $$
 3\tilde{s}^2z_2^2-6i\tilde{c}\tilde{s}z_1z_2+(7-3\tilde{c}^2)z_1^2=0.
 $$
 If $\tilde{s}\neq 0$, its roots are $z_2=\frac{6i\tilde{c}\tilde{s} z_1\pm \sqrt{84}i\tilde{s}z_1}{6\tilde{s}^2}=i\delta_\pm z_1$.\\
 Note that $\delta_-\delta_+=-\frac{7-3\tilde{c}^2}{3\tilde{s}^2}$, so neither $\delta_-$ nor $\delta_+$ can vanish.
 
 If $\tilde{s}=0$ then the set of solutions is $\{ z_1 =0\}$.
\end{proof}

This description forces the number of eigenvalues of $A|_{(\It N)^{\perp}}$.
\begin{corollary}
The shape operator $A$ of the hypersurface $H$, restricted to $(\It N)^{\perp}$, admits two distinct eigenvalues $\lambda_1$ and $\lambda_2$.
\end{corollary}

\begin{proof}
Lemma~\ref{lemma13} implies that the dimension of $d\pi(E_\lambda\cap (\mathbb J_2 N)^\perp)$ must be at most two, and since it is $\mathbb J_2$-invariant and $\mathbb J_2 N$ cannot be neither vertical nor horizontal, the dimension of $E_\lambda\cap (\mathbb J_2 N)^\perp$ is exactly two. 
\end{proof}

Next we prove that the horizontal parts of the eigenspaces are in direct sum. 
\begin{lemma}\label{lemma12}
If $N^v \neq 0$, then 
$$d\pi:T_{(x_0 , I)} \F \cap (N,\mathbb J_2N)^\perp\to T_{x_0}\CP^2 $$
is an isomorphism. In particular, as $T_{(x_0 , I)} \F \cap (N,\mathbb J_2N)^\perp$ decomposes into a direct sum of eigenspaces of $A$, we have
$$
d\pi(E_{\lambda_1})\oplus d\pi(E_{\lambda_2})=T_{x_0}\CP^2.$$
\end{lemma}
\begin{proof}
As $T \F = \hh \oplus \vv$, at a point $z = (x_0 , I) \in H \subset \F$, write $N=(N^h,N^v)$ and $\It N = (IN^h, -IN^v)$ in their horizontal and vertical components.
If $(X^h,X^v)\in T_{(x_0 , I)} \F \cap (N,\mathbb J_2N)^\perp$, we have
$$
\left\{\begin{array}{ccc}
g^h(N,X)+g^v(N,X)&=&0\\
g^h(I N,X)- g^v(I N,X)&=&0 ,
\end{array}\right.
$$
and clearly if $N^v\neq 0$ then $d\pi$ is injective.
\end{proof}

Observe that, when $N^v\neq 0$, for reasons of dimensions, the case $\tilde{s}=0$ is excluded by Lemma~\ref{lemma12}.

We fully describe $E_{\lambda}$ by obtaining its vertical part.

\begin{lemma}
Assume that neither $N^h$ nor $N^v$ vanishes, then in the basis $$\Big((N^h,I^+N^h,J^+N^h,K^+N^h),(N^v,IN^v)\Big)$$ of $T_p \F$, one of the eigenspaces of $A|_{(\It N)^{\perp}}$ is
$$
E_{\lambda}=\Vect\left(\begin{array}{c}\left(\begin{array}{c}\Vert N^v\Vert^2\\0\\0\\ \delta_+\Vert N^v\Vert^2\end{array}\right)\\ 
\\\left(\begin{array}{c}-\Vert N^h\Vert^2\\0\end{array}\right)\end{array}
,\begin{array}{c}\left(\begin{array}{c}0\\\Vert N^v\Vert^2\\-\delta_+\Vert N^v\Vert^2\\0\end{array}\right)\\ 
\\\left(\begin{array}{c}0\\-\Vert N^h\Vert^2\end{array}\right)\end{array}\right),
$$
while the other corresponds to $\delta_-$.
\end{lemma}
\begin{proof}
From Lemmas~\ref{lemma12} and \ref{lemma13}, without loss of generality, we know that
$$
d\pi(E_{\lambda})=
\Vect\left(\left(\begin{array}{c}1\\0\\0 \\\delta_+\end{array}\right),\left(\begin{array}{c}0\\1\\-\delta_+\\0\end{array}\right)\right),
 $$
 so, in the basis $\Big((N^h,I^+N^h,J^+N^h,K^+N^h),(N^v,IN^v)\Big)$, as $E_\lambda \perp \Vect (N, \It N)$, we necessarily have the above description of $E_{\lambda}$.
\end{proof}
To conclude the proof of Proposition~\ref{proposition2}, first recall that a nearly K{\"a}hler manifold satisfies~\cite{Gray1969}
\begin{equation}\label{star}
    \Vert (\nabla_X\mathbb J_2)N\Vert^2=R(X,N,X,N)+R(\mathbb J_2N,\mathbb J_2X,X,N),
\end{equation}
and, moreover, in dimension six, we have~\cite{Gray1970,Gray1976}
\begin{equation}\label{starstar}
\Vert (\nabla_X\mathbb J_2)N\Vert^2=\alpha\Vert X\Vert^2 ,
\end{equation}
where $\alpha =1$ for $\F = Z(\overline{\CP^2})$, since $\scal_{\F} = 24$ (cf.~\cite{Davidov1991}).

If $X$ is a vector field in $E_\lambda$, we have by Proposition~\ref{prop1}
\begin{align*}
&R(X,N,X,N) =R^{\CP^2} \Big(d\pi(X),d\pi(N),d\pi(X),d\pi(N)\Big)\\
&-\tfrac{11}{4}g^h(N,X)^2+\tfrac{21}{4}g^h(N,\mathbb J_2 X)^2
-\tfrac{15}{4}(\Vert N^h \Vert^2 \Vert X\Vert^2 +\Vert X^h\Vert^2)\\
&+\tfrac{11}{4}\Vert N^h \Vert^2\Vert X^h\Vert^2 +4 \Vert X\Vert^2 ,
\end{align*}
with
$$\pi^*R^{\CP^2}(X,N,X,N)=-\Vert X^h\Vert^2\Vert N^h\Vert^2+g^h(N,X)^2-3g^h(\I N,X)^2 
$$
hence
\begin{align*}
&R(X,N,X,N)=-\tfrac74 g^h(N,X)^2+ \tfrac{21}{4}g^h(N,\mathbb J_2 X)^2\\
&-3g^h(\mathbb J_{\CP^2} N,X)^2 - \tfrac{15}{4}(\Vert N^h \Vert^2\Vert X\Vert^2+\Vert X^h\Vert^2)+\tfrac74 \Vert N^h \Vert^2\Vert X^h\Vert^2 
+4\Vert X\Vert^2 .
\end{align*}
Second, observe that from page~\pageref{page8}, we have that
\begin{align*}
    &R(X,N,X,N)+R(\mathbb J_2N,\mathbb J_2X,X,N)= -g^h(\I N,X)^2-g^h(\mathbb J_2\I N,X)^2\\
&-3(\Vert N^h\Vert^2 \Vert X\Vert^2+\Vert X^h\Vert^2) 
+g^h(\mathbb J_2\I N,N)g^h(\mathbb J_2\I X,X)+4\Vert X\Vert^2.
\end{align*}
For $X\in E_\lambda$ of the form
$$ X =\left(\begin{array}{c}\left(\begin{array}{c}\Vert N^v\Vert^2\\0\\0\\ \delta_+\Vert N^v\Vert^2\end{array}\right)\\ 
\\\left(\begin{array}{c}-\Vert N^h\Vert^2\\0\end{array}\right)\end{array}\right)$$
we compute that:
\begin{align*}
&g^h(\I N,X)^2=0 ; g^h(\mathbb J_2\I N,X)^2=(-\tilde{c}+\tilde{s}\delta_+)^2\Vert N^h\Vert^4\Vert N^v\Vert^4 ; \\
&\Vert X\Vert^2=(1+\delta_+^2)\Vert N^h\Vert^2\Vert N^v\Vert^4+\Vert N^h\Vert^4\Vert N^v\Vert^2=\delta_+^2\Vert N^h\Vert^2\Vert N^v\Vert^4+\Vert N^h\Vert^2\Vert N^v\Vert^2 ;\\
&\Vert X^h\Vert^2=(1+\delta_+^2)\Vert N^h\Vert^2\Vert N^v\Vert^4 ;
g^h(\mathbb J_2\I N,N)=-\tilde{c}\Vert N^h\Vert^2 ;\\
&g^h(\mathbb J_2\I X,X)=\Big((-\tilde{c}+\tilde{s}\delta_+)+\delta_+(\tilde{s}+\tilde{c}\delta_+)\Big)\Vert N^h\Vert^2\Vert N^v\Vert^4 .
\end{align*}
Then Equation~\eqref{star} yields
$$
-4\delta_+^2\Vert N^h\Vert^4\Vert N^v\Vert^4=0,
$$
and this is impossible by the observation at the end of the proof of Lemma~\ref{lemma13}.\\
This concludes the proof of Proposition~\ref{proposition2}.
\end{proof}

We can now exclude the remaining case, since $N$ vertical has already been ruled out by Proposition~\ref{proppage6}.

\begin{proposition}\label{prop15}
Let $H$ be a hypersurface of $\F$ such that $A\varphi = \varphi A$. Then the normal vector $N$ cannot be horizontal.
\end{proposition}
\begin{proof}
If $N$ were horizontal, then by O'Neill's formula~\cite{Oneill}, for any horizontal vector $X$
$$
(AX)^v=(-\nabla_X N)^v=-\frac12 [X,N]^v ,
$$
and by~\cite{Davidov1991,deBartolomeis}, $[X,N]^v = \reallywidehat{R^{\CP^2}(X\wedge N)}$.
Let $p\in H \subset\F$, $p= (x_0, I), x_0 \in \CP^2$.
We identify vectors tangent to $\CP^2$ at $x_0$ with their horizontal lifts in $T_p \F$.
Let $(e_1,e_2,e_3,e_4)$ be an orthonormal basis of $T_{x_0}\CP^2$ adapted to our problem, i.e.
$$
e_1=d\pi(N),\quad  I=I^+\textrm{ et } \I=\tilde{c}I^-+\tilde{s}J^- ,
$$
where
\begin{center}
    \(
        \left\{
\begin{array}{ccc}
I^+&=&e_1\wedge e_2+e_3\wedge e_4\\
J^+&=&e_1\wedge e_3-e_2\wedge e_4\\
K^+&=&e_1\wedge e_4+e_2\wedge e_3
\end{array}
\right.
    \)
  \hspace{.2in}  and \hspace{.2in}
    \(
        \left\{
\begin{array}{ccc}
I^-&=&e_1\wedge e_2-e_3\wedge e_4\\
J^-&=&e_1\wedge e_3+e_2\wedge e_4\\
K^-&=&-e_1\wedge e_4+e_2\wedge e_3 .
\end{array}
\right.
    \)
\end{center}

As in the case of $\sq$, there exists a surjection 
\begin{align*}
\so (T_{x_0}\CP^2) &\to \vv_{(x_0,I)}\\
P &\mapsto \widehat P:=[I,P]=IP-PI
\end{align*}
and
$$[X,N]^v =\reallywidehat{R^{\CP^2}(X\wedge N)}.$$
One easily obtains
\begin{align*}
R^{\CP^2}(e_1,e_3)e_1&=-(1+3\tilde{s}^2)e_3-3\tilde{c}\tilde{s}e_2 ;\\
R^{\CP^2}(e_1,e_3)e_2&=(1-3\tilde{s}^2)e_4 +3\tilde{c}\tilde{s}e_1 ;\\
R^{\CP^2}(e_1,e_3)e_3&=3\tilde{c}\tilde{s} e_4 +(3\tilde{s}^2+1)e_1 ;\\
R^{\CP^2}(e_1,e_3)e_4&=-3\tilde{c}\tilde{s} e_3 - (1-3\tilde{s}^2)e_2 ,
\end{align*}
so that
$$
R^{\CP^2}(e_1\wedge e_3)=-3\tilde{c}\tilde{s} I^--(1+3\tilde{s}^2)e_1\wedge e_3+(1-3\tilde{s}^2)e_2\wedge e_4\\$$
and substituting $e_1\wedge e_3=\frac12 (J^++ J^-)$ and $e_2\wedge e_4=\frac12(J^- - J^+)$, we obtain
$$
\reallywidehat{R^{\CP^2}(e_1,e_3)}=\Big[I^+,R^{\CP^2}(e_1,e_3)\Big]=-2 K^+. 
$$
Similarly
\begin{align*}
R^{\CP^2}(e_1 \wedge e_4)&=-K^+ .
\end{align*}
As elements of $\bigwedge^+$ and $\bigwedge^-$ commute with each other
\begin{align*}
\mathcal V(A e_3)&=-\frac12\reallywidehat{R^{\CP^2}(e_1,e_3)}= K^+ ,\\
\mathcal V(A e_4)&=-\frac12\reallywidehat{R^{\CP^2}(e_1,e_4)}=- J^+ .
\end{align*}
In the basis $\Big(\{e_2\},\{e_3,e_4\},\{J^+,K^+\}\Big)$, the endomorphisms $A$ and $\varphi$ have the following block matrices
$$
A=\left(\begin{array}{ccc}\mu &0&0\\
&&\\
0&E&F\\
&&\\
0&F&G\end{array}\right)
\quad\textrm{ with
$F=\left(\begin{array}{cc}
0&-1\\
1&0
\end{array}
\right)
$, }
$$
and
$$
\varphi=\left(\begin{array}{ccc}0&0&0\\0&I&0\\0&0&-I\end{array}\right) \textrm{ with } I=\left(\begin{array}{ccc}0&-1\\1&0\end{array}\right)
.$$
By hypothesis, $A$ and $\varphi$ commute, which contradicts the form of block $F$.
\end{proof}



\begin{thebibliography}{99}

\bibitem{Apostolov98}
\textsc{V.~Apostolov, G.~Grantcharov and S.~Ivanov},
\newblock Hermitian structures on twistor spaces,
\newblock Ann. Global Anal. Geom. 16 (1998), 291--308.

\bibitem{Atiyah}
\textsc{M.~F. Atiyah, N.~J. Hitchin and I.~M. Singer},
\newblock Self-duality in four-dimensional {R}iemannian geometry,
\newblock Proc. Roy. Soc. London Ser. A 362 (1978), 425--461.

\bibitem{Berndt1995}
\textsc{J.~Berndt, J.~Bolton and L.~M. Woodward},
\newblock Almost complex curves and {H}opf hypersurfaces in the nearly
  {K}\"{a}hler 6-sphere,
\newblock Geom. Dedicata 56 (1995), 237--247.

\bibitem{Blair1971}
\textsc{D.~E. Blair},
\newblock Almost contact manifolds with {K}illing structure tensors,
\newblock Pacific J. Math. 39 (1971), 285--292.

\bibitem{Butruille2005}
\textsc{J.-B. Butruille},
\newblock Classification des vari{\'{e}}t{\'{e}}s approximativement
  k\"{a}hl{\'{e}}riennes homog{\`{e}}nes,
\newblock Ann. Global Anal. Geom. 27 (2005), 201--225.

\bibitem{Cecil2015}
\textsc{Th.~E. Cecil and P.~J. Ryan},
\newblock Geometry of Hypersurfaces,
\newblock Springer New York, 2015.

\bibitem{Davidov1991}
\textsc{J.~Davidov and O.~Muskarov},
\newblock On the {R}iemannian curvature of a twistor space,
\newblock Acta Mathematica Hungarica 58 (1991), 319--332.

\bibitem{deBartolomeis}
\textsc{P.~de~Bartolomeis and A.~Nannicini},
\newblock Introduction to differential geometry of twistor spaces,
\newblock In {\em Geometric theory of singular phenomena in partial
  differential equations ({C}ortona, 1995)}, Sympos. Math., XXXVIII, pages
  91--160. Cambridge Univ. Press, Cambridge, 1998.

\bibitem{Eells-Salamon}
\textsc{J.~Eells and S.~Salamon},
\newblock Twistorial construction of harmonic maps of surfaces into
  four-manifolds,
\newblock Ann. Scuola Norm. Sup. Pisa Cl. Sci. 12 (1985), 589--640.

\bibitem{Friedrich}
\textsc{Th. Friedrich and H.~Kurke},
\newblock Compact four-dimensional self-dual {E}instein manifolds with positive
  scalar curvature,
\newblock Math. Nachr. 106 (1982), 271--299.

\bibitem{Gray1969}
\textsc{A.~Gray},
\newblock Almost complex submanifolds of the six sphere,
\newblock Proc. Amer. Math. Soc. 20 (1969), 277--277.

\bibitem{Gray1970}
\textsc{A.~Gray},
\newblock Nearly {K}\"{a}hler manifolds,
\newblock J. Differential Geom. 4 (1970), 283--309.

\bibitem{Gray1976}
\textsc{A.~Gray},
\newblock The structure of nearly {K}\"{a}hler manifolds,
\newblock Math. Ann. 223 (1976), 233--248.

\bibitem{Hitchin}
\textsc{N.~J. Hitchin},
\newblock K\"{a}hlerian twistor spaces,
\newblock Proc. London Math. Soc. 43 (1981), 133--150.

\bibitem{Hu2017}
\textsc{Z.~Hu, Z.~Yao and Y.~Zhang},
\newblock On some hypersurfaces of the homogeneous nearly {K}\"{a}hler
  $\mathbb{S}^3 \times \mathbb{S}^3$,
\newblock Math. Nachr. 291 (2017), 343--373.

\bibitem{Martins2001}
\textsc{J.~Kenedy Martins},
\newblock Congruence of hypersurfaces in $\sn^6$ and in
  $\mathbb{C}\mathrm{P}^n$,
\newblock Bull. Braz. Math. Soc. 32 (2001), 83--105.

\bibitem{Montiel}
\textsc{S.~Montiel},
\newblock Real hypersurfaces of a complex hyperbolic space,
\newblock J. Math. Soc. Japan 37 (1985), 515--535.

\bibitem{Moruz2018}
\textsc{M.~Moruz and L.~Vrancken},
\newblock Properties of the nearly {K}\"{a}hler $\mathbb{S}^3 \times
  \mathbb{S}^3$,
\newblock Publ. Inst. Math. (Beograd) 103 (2018), 147--158.

\bibitem{Mush}
\textsc{O.~Mu\v{s}karov},
\newblock Structures presque hermitiennes sur des espaces twistoriels et leurs
  types,
\newblock C. R. Math. Acad. Sci. Paris 305 (1987), 307--309.

\bibitem{Nagy2002}
\textsc{P.-A. Nagy},
\newblock Nearly {K}\"{a}hler geometry and {R}iemannian foliations,
\newblock Asian J. Math. 6 (2002), 481--504.

\bibitem{Oneill}
\textsc{B.~O'Neill},
\newblock The fundamental equations of a submersion,
\newblock Michigan Math. J. 13 (1966), 459--469.

\bibitem{Takagi1}
\textsc{R.~Takagi},
\newblock Real hypersurfaces in a complex projective space with constant
  principal curvatures,
\newblock J. Math. Soc. Japan 27 (1975), 43--53.

\bibitem{Takagi2}
\textsc{R.~Takagi},
\newblock Real hypersurfaces in a complex projective space with constant
  principal curvatures. {II},
\newblock J. Math. Soc. Japan 27 (1975), 507--516.

\end{thebibliography}
\end{document}